\documentclass[11pt,leqno]{amsart}
\usepackage[latin1]{inputenc}
\usepackage[T1]{fontenc}
\usepackage{bbm}
\usepackage{amsmath}
\usepackage{amsmath,amsfonts,amssymb,amsthm} 
\usepackage{amsthm}
\usepackage{latexsym}
\usepackage{amssymb}
\usepackage[all]{xy}
\usepackage{calc}
\usepackage{multicol}
\usepackage{color}
\usepackage{bbm}
\newtheorem{thm}{Theorem}[section]
\newtheorem{prop}{Proposition}[section]

\newtheorem{lem}{Lemma}[section]

\usepackage{appendix}

\newcommand{\R}{\mathbb{R}}
\newcommand{\N}{\mathbb{N}}
\newcommand{\Z}{\mathbb{Z}}

\newcommand{\be}{\begin{equation}}
\newcommand{\ee}{\end{equation}}

\newcommand{\T}{\mathbb{T}}

\usepackage[active]{srcltx}

\newcommand{\Lap}{\mathop{\kern0pt \ms{L}}\mathopen{}}

\newcommand{\Hess}{\mathop{\kern0pt \mathrm{D}^{2}}\mathopen{}}
\newcommand{\LP}{\mathop{\kern0pt \mathrm{P}}\hspace{-0.25em}\mathopen{}}
\newcommand{\QP}{\mathop{\kern0pt \mathrm{Q}}\mathopen{}} 
\newcommand{\UP}{\mathop{\kern0pt \mathrm{U}}\hspace{-0.15em}\mathopen{}}
\newcommand{\dec}{\mathop{\kern0pt \mathrm{dec}}\mathopen{}} 

\newcounter{exercice}

\newcounter{mycounter}

\usepackage{tikz}

\makeatletter
\newcommand\RSloop{\@ifnextchar\bgroup\RSloopa\RSloopb}
\makeatother
\newcommand\RSloopa[1]{\bgroup\RSloop#1\relax\egroup\RSloop}
\newcommand\RSloopb[1]%
  {\ifx\relax#1%
   \else
     \ifcsname RS:#1\endcsname
       \csname RS:#1\endcsname
     \else
       \GenericError{(RS)}{RS Error: operator #1 undefined}{}{}%
     \fi
   \expandafter\RSloop
   \fi
  }
\newcommand\X{0}
\newcommand\RS[1]%
  {\begin{tikzpicture}
     [every node/.style=
       {circle,draw,fill,minimum size=1.5pt,inner sep=0pt,outer sep=0pt},
      line cap=round
     ]
   \coordinate(\X) at (0,0);
   \RSloop{#1}\relax
   \end{tikzpicture}
  }
\newcommand\RSdef[1]{\expandafter\def\csname RS:#1\endcsname}
\newlength\RSu
\RSdef{i}{\draw (\X) -- +(90:\RSu) node{};}
\RSdef{l}{\draw (\X) -- +(135:\RSu) node{};}
\RSdef{r}{\draw (\X) -- +(45:\RSu) node{};}
\RSdef{I}{\draw (\X) -- +(90:\RSu) coordinate(\X I);\edef\X{\X I}}
\RSdef{L}{\draw (\X) -- +(135:\RSu) coordinate(\X L);\edef\X{\X L}}
\RSdef{R}{\draw (\X) -- +(45:\RSu) coordinate(\X R);\edef\X{\X R}}

\usepackage{mathabx}

\makeatletter
\newcommand\incircbin
{%
  \mathpalette\@incircbin
}
\newcommand\@incircbin[2]
{%
  \mathbin%
  {%
    \ooalign{\hidewidth$#1#2$\hidewidth\crcr$#1\ovoid$}%
  }%
}

\makeatother

\begin{document}

\title[
Invariant measure NLS]{New invariant surface measures for the cubic Non Linear Schr\"odinger equation}
\author[J.-B. Casteras]{Jean-Baptiste Casteras}
\address{CEMSUL, Faculdade de Ci\^encias da Universidade de Lisboa, Edificio C6, Piso 1, Campo Grande 1749-016 Lisboa, Portugal}
\email{jeanbaptiste.casteras@gmail.com}
\author[A. B. Cruzeiro]{Ana Bela Cruzeiro}
\address{GFM, Dep. de Matem\' atica, Instituto Superior T\' ecnico, Lisboa, Portugal}
\email{ ana.cruzeiro@tecnico.ulisboa.pt}
\author[A. Millet]{Annie Millet}
\address{SAMM, EA 4543, Universit\' e Paris 1 Panth\' eon Sorbonne, 90 Rue de Tolbiac, 75634 Paris
Cedex France and Laboratoire de Probabilit\' es, Statistique et Mod\' elisation, UMR 8001,
Sorbonne Universit\' e - Universit\'e Paris Cit\'e.}
\email{ annie.millet@univ-paris1.fr}
\thanks{J.-B.C. is supported by FCT - Funda\c c\~ao para a Ci\^encia e a Tecnologia, under the project: UID/04561/2025. A. B. C.  is supported by FCT - Funda\c c\~ao para a Ci\^encia e a Tecnologia, under the project: UIDB/00208/2020. The work of A.M. is conducted within the FP2M federation (CNRS FR 2036).}

\begin{abstract}
We construct new invariant measures supported on mass level sets for the cubic defocusing nonlinear Schr\" odinger equation
 on the torus in dimensions $1$ and $2$. 
\end{abstract}

\maketitle

\section{Introduction}
 Consider the following Non Linear Schr\"odinger (NLS) equation for $\lambda \in \{ -1, +1\}$ and $p>2$
\begin{equation}
\label{eqintro}
i\frac{\partial v}{\partial t}= -\Delta v +  \lambda |v|^{p-2} v,
\end{equation}
 on the torus $\T^d$ of dimension $d$ (where $\T={\R}/{\Z}$), with initial conditions $v_0\in H^s$, for some $s\in \R$. When $\lambda =1$ (resp. $\lambda =-1$), the equation is defocusing
 (resp. focusing). Recall that \eqref{eqintro} can be seen as a Hamiltonian PDE, i.e.,
$$\dfrac{d u}{dt}= i \dfrac{\partial H}{\partial \overline{u}},$$
where 
$$H(u)= \dfrac{1}{2} \int_{\T^d} |\nabla u|^2 dx+  \dfrac{\lambda}{p} \int_{\T^d} |u|^{p} dx.$$
From this reformulation, it is easily seen that $H(u)$ is an invariant quantity. One can also check that the mass 
$$M(u)= \int_{\T^d} |u|^2 dx,$$
and the momentum
$$P(u)=i\int_{\T^d} u \overline{u}_x dx$$
 are conserved. In general, these three conservation laws are the only known ones.  But in some very special cases, others are available. 
For instance, if $d=1$, the cubic Schr\" odinger equation possesses a Lax pair structure which allows to construct an infinity of conservation laws see \cite{MR905674}. 
In this case, the equation is said to be completely integrable. 

In the case equation \eqref{eqintro} is satisfied on the entire space ${\mathbb R}^d$ instead of the torus, the NLS equation  has been intensively
studied (local and global well-posedness, scattering, blow-up time and profile, ...). Most results suppose that the initial condition belongs to $H^1({\mathbb R}^d)$ and the proofs use 
 Strichartz's and Gagliardo-Nirenberg's inequalities. 
\bigskip

The equation \eqref{eqintro}   on the torus $\T^d$ has been intensively studied as well.  Versions of the Strichartz inequality (with unavoidable loss of regularity) have been obtained
(see e.g. \cite{MR3512894} and the references therein).
Using some new version of the Strichartz inequality, Herr and Kwak  \cite{herr2024strichartz} 
have proved that the cubic NLS equation
on the torus ${\mathbb T}^2$ is globally well-posed in  $H^s$, $s>0$, for any initial condition in the defocusing case (and under some "small data" assumptions in the focusing case); see also \cite{MR3512894}, \cite{kwak2024critical}.

Starting from the work of Lebowitz, Rose and Speer \cite{LRS}, and later on from the papers of Bourgain \cite{B93, B94,B96}, the statistical approach to solve the non linear  Schr\"odinger equation
 on the $d$-dimensional torus has been considerably developed. 
 It consists in constructing an invariant probability (Gibbs) measure for the solutions of these equations, based on their invariant quantities.
 Starting almost everywhere from the support of the measure,  solutions are proved to exist and be unique.  They  can be rather singular, typically living in Sobolev spaces of low regularity. 
\smallskip

 First, let us describe in a few words the situation in the $d$-dimensional case. 
 Let $H(p,q)=H(p_1 ,\ldots ,p_d , q_1 ,\ldots ,q_d)$ be a Hamiltonian and let us consider its associated flow in $\R^{2d}$, i.e., 
\begin{equation}
\label{Hamintro}
\dot{p}_i = \dfrac{\partial H}{\partial q_j},\ \dot{q}_i=-\dfrac{\partial H}{\partial p_j}, \quad  i,j=1, ...,d.
\end{equation}
Then, using the conservation of $H$, Liouville's theorem states that the Gibbs measure $e^{-H(p,q)} \Pi_{j=1}^d dp_j dq_j$ is invariant under the flow generated by 
\eqref{Hamintro}. Notice that, for any 
function $F(p,q)$ which is conserved by the flow \eqref{Hamintro}, the measure $F(p,q) e^{-H(p,q)} \Pi_{j=1}^d dp_j dq_j $ is also invariant. 
The situation is much more delicate in the infinite dimensional case. 
 In analogy with the Euclidean case $\R^d$, by a scaling argument, one expects \eqref{eqintro} to be (at least locally) well-posed for an initial data in $H^{s}$, for $s>s_c= \dfrac{d}{2}-1$. 
The space $H^{s_c}$ is called critical.
\smallskip

In dimension 1, in \cite{B93,B94}, J. Bourgain proved the existence of an invariant measure $\mu$ for equation \eqref{eqintro} for $p\in [4,6]$ in appropriate $H^s$ spaces, and proved the global
well-posedness of the equation for an initial condition in a set of full $\mu$ measure.
 The proof uses a truncated Fourier series, some version of the Strichartz inequality, the "pigeon hole principle" and a large deviations argument. 
\bigskip

 When $d=2$, J.~Bourgain \cite{B93} obtained the local well-posedness of \eqref{eqintro} for initial data in $H^s$, $s>0$. 
The global well-posedness in this Sobolev space was obtained in \cite{herr2024strichartz}. 
In dimensions $3$ and $4$, for  energy-critical non-linearity,  local well posedness was proved  in $H^1$, while
global well-posedness was shown for "small" initial data; see \cite{MR3512894} and the references therein. 
We also refer to \cite{Sta} for related results based on invariant Gibbs measures.
For  the cubic non-linear Schr\"odinger equation (NLS), in the defocusing case and on the two-dimensional torus, a renormalisation must be introduced. 
This equation was studied in \cite{B96}, with the construction of a Gibbs measure on $\bigcap_{s<0} H^s(\T^2)$ associated with the Wick ordered Hamiltonian. 
In \cite{DNY}, the authors study a nonlinearity of the form $|v|^{p-2} v$ for an even integer $p\geq 4$. They also prove
  the existence and uniqueness of the corresponding flow defined on the
 support of the measure by developping the so-called theory of random tensors.
\bigskip

Our goal in this paper is to construct new invariant measures $\mu$ for the defocusing cubic NLS equation in dimension 2 under some mass constraint with support in
$\bigcap_{s<0} H^s(\T^2)$. 
Our technique gives a similar result in dimension 1 which is simpler and does not require renormalization. 
This approach uses techniques from probability theory, where transport properties of Gaussian measures under linear and nonlinear transformations have an old tradition.
 In the framework of Hamiltonian partial differential equations, of which Schr\"odinger systems are an example,
  invariant (or quasi-invariant) measures are typically associated to conservation laws and are weighted Gaussian measures.

Concerning invariant measures with prescribed values of invariant quantities, we mention two works, in which not only the measures but also the methods are quite different from ours.
We refer to the work \cite{OQ} of  T.~Oh and J. ~Quastel, who have proved 
existence of  invariant measures on $H^{\frac{1}{2}-\gamma}$, for some positive $\gamma$, conditioned on fixed values $a>0$ for the  mass and $b\in \R$ 
for the momentum for focusing/defocusing NLS when $p\in [4,6]$  
(with some smallness condition on the mass for a focusing 
quintic nonlinearity, that is $p=6$).

These measures have the form

$$d\mu_{a,b} =\frac{1}{Z} \mathbbm{1}_{\{\int |u|^2 =a\}}\mathbbm{1}_{\{i\int u\partial_x \bar u =b \}}e^{\pm \frac{1}{p} \int |u|^p -\frac{1}{2} \int |u|^2 }dP$$
for $P$ a complex-valued Wiener measure on the circle.

J.~Brereton \cite{MR3900225}   proved the existence of an invariant measure on $\bigcap_{s>1/2} H^s(\T)$ conditioned on mass for the cubic derivative NLS equation, that is when $|v|^{2} v$
is replaced by $i \partial_x(|v|^2 v)$. The construction is made via an infinite-dimensional version of the divergence theorem.

The Gibbs measures are absolutely continuous with respect to Gaussian ones. We consider the corresponding abstract Wiener spaces in the sense of Gross \cite{G}. 
Then we use techniques from quasi-sure analysis on these spaces, as established by Malliavin \cite{M}. For a non degenerate functional on an abstract Wiener space we can 
consider its level sets, which are (finite codimensional) submanifolds of that space. In \cite{AM} Airault and Malliavin defined a disintegration of the Gaussian measure
 into a family of conditional laws or surface measures and proved  a geometric Federer-type coarea formula providing  this disintegration.
Such measures do not charge sets of capacity zero.

 In this work we study the defocusing cubic renormalised NLS on the two dimensional torus and construct invariant surface measures for the mass, 
which is a conserved quantity of the equation in the same spirit as in \cite{B96,OQ}. Then we prove existence and uniqueness of the associated flows, living on  level sets of the mass. Therefore, we can start with initial conditions distributed on sets which are thinner than the support of the Gibbs measures. Let us observe that \cite{B96,OQ} were proved in dimension $1$ and therefore were not requiring normalisation (which makes computations a lot harder). It seems our result is the first one holding in higher dimension. Let us also point out that \cite{B96,OQ} are not interested in the global existence of solution to the Schr\" odinger equation but only in the measure.  


Our methods are similar to those in \cite{C}, where surface measures supported by level sets of the enstrophy were constructed for the Euler equation on the torus.

More precisely, we study the cubic defocusing Schr\" odinger equation in the one and two dimensional torus, namely

\begin{equation}	\label{cubic-defoc}
i\frac{\partial v}{\partial t}= -\Delta v +|v|^2 v
\end{equation}
with initial conditions $v_0$ of low regularity. If the solution is assumed to be at least $L^2$, we can use the following change of variables $v=e^{it (1-2 \|u\|^2_{L^2} )} u$. In this case, 
\eqref{cubic-defoc} is equivalent to 
\begin{equation}
\label{rNLS}
i\frac{\partial u}{\partial t}= -\Delta u+ \Big( |u|^2 -2  \|u \|_{L^2}^2 \Big)u +u .
\end{equation}
Notice that it is not anymore the case if the solution does not have a finite mass.
 Our aim is to take initial data in the support of a surface measure supported on mass level sets. Formally, let
$$  d\mu_{2s}(u) = Z_s^{-1} e^{-\frac{1}{2}\|u\|_{H^s}^2} du,$$
where $Z_s$ is a normalisation constant and $du$ is formally the Lebesgue measure,  which is well-known to not exist on infinite dimensional vector spaces. 
 The Gaussian measures $\mu_{2s}$ can also be seen as the laws of the random variables 
$$\omega \rightarrow \sum_{n\in \Z^d} \dfrac{g_n (\omega)}{(1+|n|^2)^{s/2}} e^{inx},$$
where $\{g_n\}$ are independent standard complex-valued Gaussian variables on a probability space $(\Omega , \mathcal{A},\mathbb{P})$. For $\sigma \in \R$, it is easy to check that
$$\mathbb{E} \Big[\sum_{n\in \Z^d} (1+|n|^2)^{\sigma} \Big|\dfrac{g_n}{(1+|n|^2)^{s/2}} \Big|^2 \Big]<\infty \ ~\hbox{iff}~ \ \sigma<s-\dfrac{d}{2}.$$
So this random series converges in $\mathbb{L}^2 (\Omega , H^\sigma (\T^d))$ if and only if $\sigma < s- \dfrac{d}{2}$. In what follows, we will always take $s=1$. 
We see that, when $d=2$, solutions distributed according to $\mu_2$ do not have a finite mass, which implies that equation \eqref{cubic-defoc} and \eqref{rNLS} are not equivalent. 
The previous remarks also show that the support of  this Gibbs measure is necessarily rough. We point out that an alternative method to construct invariant measures 
(with possibly smooth support) is the so-called Inviscid-Infinite-dimensional limits method due to Kuksin and Shirikyan \cite{MR2039838,MR3443633}. 
We refer to \cite{MR4709547,MR4312285,MR4503168} 
for applications related to NLS.

We set $V_r = \{\varphi \in L^2 | E(\varphi)=r\}$, for $r>0$ where
\begin{equation}
\label{lastdefE}
E(\varphi)=\int_{\T^d } |\varphi|^2 dx - \mathbbm{1}_{\{d=2\} } \; \sum_{k\neq 0} \dfrac{2}{|k|^{2}} . \end{equation}

In section \ref{renormalized}, we prove that this quantity is well-defined in dimension $2$. Our main result reads as follows :

\begin{thm}	\label{main-intro}
Let $d=1$ or $d=2$. Then, for all $r>0$, there exists a probability measure $\sigma_r$ defined on $L^2$ with support on $V_r$ and a flow $u_t$ solution of (1) if $d=1$ or (4) of $d=2$, defined $\sigma_r$-a.e. for all $t\in \mathbb R$ with respect to which $\sigma_r$ is invariant, namely
$$\frac{d (u_t )_\ast \sigma_r}{d\sigma_r}= 1,  \quad  \forall t\in \R .$$
\end{thm}
See Theorem \ref{Th5.1} for a more precise statement.

Let us also point out that our result can be easily generalised to more general dissipative equations, 
namely in \eqref{eqintro}, we could have taken any power of the Laplace operator
 $(-\Delta)^s$, $s>0$ as well as any odd power nonlinearities. In the case of a more general dispersion of the form 
 $(-\Delta)^s,$ the difference would have been the value of the parameter $a=2$ in the above renormalisation to define $E(\varphi)$ (see \eqref{lastdefE}).
 Instead of taking it equal to $2$, it would have been chosen as $a=2s$.

In section \ref{preliminaries}, we describe the Gaussian measures we will use, define the Sobolev and Wiener spaces.
Finally, we state a result by  Airault and Malliavin \cite{AM} which constructs a surface measure supported by a level set  of 
some  functional defined on a Wiener space; this is the main ingredient in the construction of our invariant measure.
In section \ref{estim-B}, we prove Sobolev estimates on the non linear term projected on low frequencies; in dimension 2 this will require a renormalisation of this cubic term. 
A similar renormalisation was also considered in \cite{B96}. In Section $4$, we prove several estimates for a renormalisation of the mass. 
Finally, we prove our main result  in Section $5$. Some technical lemma is proved in the appendix. 

\vskip 12mm

\section{Preliminaries}	\label{preliminaries} 
Let us first  state several basic facts about Gaussian measures. 
Next, we rewrite the cubic Schr\" odinger equation into frequency. 

\subsection{The Gaussian measures}
 Given a function $\varphi \in L^2({\mathbb T}^d)$, set $\varphi = \sum_{k\in \Z^d} \varphi_k e_k$, where 
  $e_k (x)= \frac{1}{2} e^{ikx}$, $k\in \Z^d$, with $kx=k_1 x_1+\ldots  +k_d x_d$.

For $a>0$ let us consider  the Gaussian measures on $L^2({\mathbb T}^d)$
$$d\mu_a = \Pi_{k\in \Z^d} d \mu_a^k ,\quad 
d\mu_a^k (z)= \frac{|k|^{a}}{2\pi}e^{-\frac{1}{2} |k|^a |z|^2} dx dy,$$
with $z=x+iy \in \mathbb{C}$,  $k= (k_1,\ldots , k_d) \in \Z^d$ and $|k|^2 = k_1^2+\ldots + k_d^2$.

Suppose that $\{ \varphi_k , k\in \Z^d\}$ are independent  complex-valued 
Gaussian random variables with distribution $\mu_a^k$. 
For later purpose, let us notice that
$$E_{\mu_a}(\varphi_k)=0,\ E_{\mu_a} (\varphi_k \varphi_{k^\prime})=0,\ E_{\mu_a} (\varphi_k \overline{\varphi_{k^\prime}})=\delta_{k,k^\prime}  \; \frac{2}{|k|^{a}},$$
and, for any $p\geq 1$,
\begin{equation}
\label{formvar}
\mathbb{E}_{\mu_a} (|\varphi_k|^{2p})=\dfrac{c_p}{|k|^{ap}}, \quad \mbox{\rm where} \quad 
c_p= 2^p p!.
\end{equation} 
Indeed, letting $x=Re (\varphi_k)$, $y= Im (\varphi_k)$, 
we have
\begin{align*}
\mathbb{E}_{\mu_a} (|\varphi_k|^{2p}) &=
\frac{ |k|^a}{2\pi}  \int_{\R \times \R}(x^2 +y^2 )^p e^{- |k|^a \frac{x^2 +y^2}{2}} dx dy\\
&=  |k|^{a} \int_{\R^+ } r^{2p} e^{- |k|^a \frac{r^2}{2}}r dr\\
&=  2^p p!  \,  |k|^{-ap}. 
\end{align*}

\subsection{The Sobolev spaces} We next  recall some definitions of  functional spaces such as Sobolev spaces and Wiener spaces. 
The Sobolev spaces of order $\beta$ on the torus are defined by

$$H^{\beta} =\Big\{  \varphi \in L^2  : \sum_{  k\neq 0} |k|^{2\beta} |\varphi_k |^2 < +\infty \Big\}$$
and are endowed with the corresponding Hilbert scalar product.

We will work with the Gaussian measure $\mu_2$. This measure is supported in the space $H^{\beta}$ with $\beta <0$. 
This is a consequence of the fact that
$\int \sum_{ k\neq 0}  |k|^{2\beta} |\varphi_k |^2 d\mu_2 =2\sum_{ k\neq 0} \frac{|k|^{2\beta}}{|k|^2} $ converges for $\beta <0$; 
it is also a consequence of $\mu_2$ being the measure induced by  the process 
$\sum_{|n|>0}  \frac{1}{n} G_n e_n (x)$, where $\{G_n\}_n$ are independent standard complex Gaussian random variables.

 For $\beta <0$,  the triple $(H^{\beta}, H^1, \mu_2 )$ is an abstract Wiener space in the sense of Gross (see \cite{G}).

Given a Hilbert space $G$, we consider the Malliavin derivative of a functional $F: H^{\beta} \rightarrow  G$,  in the Cameron-Martin direction $v\in H^1$, namely the (a.e.)  limit 
$$\lim_{\varepsilon \rightarrow 0} \frac{1}{\varepsilon} (F(\varphi +\varepsilon  v ) -F(\varphi )), 
\quad \forall  \varphi \in H^\beta, \forall  v\in H^1 .$$
By the Riesz theorem this derivative gives rise to a linear gradient operator $\nabla$ such that the limit above coincides with $<\nabla F (\varphi ), v>$. 
One can iterate this procedure and obtain higher order gradients $\nabla^r$. Let us for instance consider the second derivative of $F$. For $\varphi \in H^\beta$, and $v,w \in H^1$,
 iterating the first derivative, we obtain
$$ \nabla^2 F(\varphi)(v,w)= D_v D_w F(\varphi).$$
We define the Hilbert-Schmidt norm
$$\|\nabla^2 F(\varphi)\|^2_{H.S.(H^1 \otimes H^1 , H^\beta)} =\sum_{j,k} \|D_{\hat{e}_j} D_{\hat{e}_k} F(\varphi )\|^2_{H^\beta},$$
where $\{\hat{e}_j\}_j$ is an orthonormal basis of $H^1$. The corresponding Sobolev spaces are 
$W_r^p (H^\beta; G)$, that is 
 the space of maps $F: H^\beta \rightarrow G$ such that 
$F\in L_{\mu_2}^p (H^\beta ; G)$, the gradients $\nabla^s F : H^\beta \rightarrow H.S. (\otimes_s H^1 , G)$ are defined for every $1\leq s \leq r$ and belong to 
$L_{\mu_2}^p (H^\beta ; H.S. (\otimes_s H^1 , G))$. 
We equip this space with the norm
$$\|f\|_{W_r^p (H^\beta ; G)} = \|f \|_{L_{\mu_2}^p (H^\beta ; G) } + \sum_{1\leq s \leq r} \|f\|_{L_{\mu_2}^p (H^\beta ; H.S. (\otimes_s H^1 , G) }.$$
We also denote $W_{\infty} = \bigcap_{r,p} W_{r}^p (H^\beta ; \mathbb{C})$ 
 and $W_{\infty} (H^\beta) = \bigcap_{r,p} W_{r}^p (H^\beta ; H^\beta)$.

\subsection{The nonlinear Schr\"odinger equation}
  Using the decomposition of the solution $\varphi(x,t)=\sum_k \varphi_k(t) e_k(x)$  into frequencies, the cubic Schr\" odinger equation \eqref{cubic-defoc}  for $\varphi$ 
 can be rewritten as the family of equations 
$$\dfrac{d}{dt}\varphi_k = -i |k|^2 \varphi_k -i B_k (\varphi), \quad |k|>0, $$
 where 
$$B_k (\varphi)=  \sum_{l,m} \varphi_{k-m} \varphi_{l+m} \overline{\varphi}_l .$$
The functions $\{e_k\}_k$ are eigenfunctions for the operator $-\Delta$ with eigenvalues $|k|^2 =k_1^2+\ldots  +k_d^2$; 
 we set $A_k (\varphi )=-|k|^2 \varphi_k$, corresponding to the Laplacian $A=\Delta$.

For any $n\in \N$, we define 
$$B^n (\varphi)= \Pi_n (B(\Pi_n (\varphi))),$$
where 
$$\Pi_n \Big(\sum_{k} \varphi_k e_k\Big) = \sum_{k: 1\leq |k|^2 \leq n} \varphi_k e_k. $$

\subsection{Surface measures}
When $g\in W_\infty$, it is possible to  consider a redefinition of  $g$,  namely a function $g^\ast$ that coincides with $g$ almost surely   for $\mu_2$, 
 and is defined and continuous 
outside sets of capacity zero (c.f \cite{M}). Using such redifinitions a construction of surface measures on Wiener spaces was made in \cite{AM}. 
These measures are defined on level sets of $W_\infty$ functionals,  and do not charge zero  capacity sets. 

Let $E$ be a real-valued functional defined on $H^{\beta}$ and consider the surfaces $V_r =\{\varphi | E(\varphi)=r\}$ for $r>0$.
Assume that $E$ and $\| \nabla E \|^{-1}_{H^1}$ belong to the corresponding spaces $W_\infty$.

Denote the $C^\infty$ densities of $d(E \ast \mu_2)$ (resp. $d(E \ast g \mu_2)$) with respect to the Lebesgue measure by $\rho (r)= \dfrac{d (E\ast \mu_2)}{dr}$ 
(resp. $\rho_g (r)= \dfrac{d (E\ast g\mu_2)}{dr}$); they  are well defined and smooth, cf. \cite{AM}.
We have 
\begin{thm} \cite{AM}
\label{thmredi}
Let $r>0$ be such that $\rho (r)>0$. Then there exists a Borel probability measure defined on $H^\beta$, denoted by $\nu^r$, with support on $V_r$ and such that
$$\int_{V_r} g^\ast (\varphi) d\nu^r = \dfrac{\rho_g (r)}{\rho (r)},$$
for any $g^\ast$ redifinition of $g$.
\end{thm}
This will be a crucial ingredient to construct our invariant measures.

\section{Norm estimates for $B$ }  \label{estim-B}

Our goal in this section is to derive Sobolev estimates for the vector field $B$. 
We will see that we need to renormalize $B$ when the dimension $d=2$. We proceed as follows: for every $N\in \N$  and $k\in {\mathbb Z}^2$, we let

\begin{equation} 	\label{:B:}
(:B:)_k^N= B_k - 4 \; \Big(\sum_{m:1\leq |m| \leq |N|} \frac{1}{|m|^a}\Big) \varphi_k,
\end{equation}
with $a=2$. Notice that this sum diverges when $|N| \rightarrow \infty$ if $d=2$. In dimension $1$, a renormalisation is not needed (since the sum is converging).
 More precisely our aim in this section is to show that in dimension 2 $(:B:)^N$ belongs to $ L^{2p}_{\mu_a} (H^\beta , H^\beta)$ and $\nabla^{s}( :B:)^N$ 
 belongs to $L^{2p}_{\mu_a} (H^\beta , H.S. (\otimes^{s} H^1 ,H^\beta))$, for any $p\geq 1$, $s\in \N^\ast$ and $\beta <0$, uniformly in $N$.
  Thanks to our estimate, we will be able to prove that $(:B:)^N$ converges in $ L^{2p}_{\mu_a} (H^\beta , H^\beta)$. We will denote its limit by $:B:$. We begin by estimating $(:B:)^N$.

\begin{prop}
\label{prop1}
 Let $a=2$ and $\beta <0$ and $p\geq 1$ an integer.

In dimension 1,   $B$ is $L^p$-integrable with respect to the measure $\mu_a$ and to the norm $H^\beta$.

In dimension 2,  $(:B:)^N$  is $L^p$-integrable with respect to the measure $\mu_a$ and to the norm $H^\beta$, uniformly in $N$. 
We denote by $:B:$  its limit  as $N\to \infty$.

\end{prop}
Notice that our choice of parameters is related to the fact that we want to work on the abstract Wiener space $(H^\beta ,H^1 ,\mu_2)$, with $\beta<0$.
\begin{proof}
By abuse of notation, in this proof, we will drop the subcript $N$ in $:B:^N$ in dimension 2.  We will work at fixed $N$ and will notice at the end that, thanks to the normalisation,  
 our estimates do not depend on $N$. 
 
Let $p$ be odd; we are going to show that $:B:\, \in L^{2p}_{\mu_a} (H^\beta , H^\beta)$.
In a first step, we  compute  
$E_{\mu_a}  (|B|^{2p} )$.  
\smallskip

For every $k\in \Z$ with $k\neq 0$ we have 
$$|B_k (\varphi)|^2 =  \sum_{l,l'\neq 0, m,m', |l|\vee |m|\vee |l'|\vee |m'| \leq N}T_k(l,m,l',m'), $$
 where, for $m\not\in \{k, -l\}$ and $m'\not\in \{ k, -l'\}$,  
\[ T_k(l,m,l',m'):= \varphi_{k-m} \varphi_{l+m}\varphi_{l'}  \overline{\varphi_l}  \overline{\varphi_{k-m'}} \overline{\varphi_{l'+m'}}.\]
Note that $E_{\mu_a} [T_k(l,m,l',m')^{p} ]= 0$ except in three cases: $k-m=l$, $k-m=k-m'$ and $k-m=l'+m'$.\bigskip

\noindent {\it Case 1}\; Let  $k-m=l$;  then we have to impose that either $l+m=k-m'$ or $l+m=l'+m'$.

 {\it Case 1.1} \; If $l+m=k-m'$, we deduce that, since $k-l=m=m+m'$, we have $m'=0$. Therefore, $T_k(l,k-l,l',0)= |\varphi_k|^2 |\varphi_l|^2 |\varphi_{l'}|^2$. 
Thus, 
\begin{itemize} 
\item if  $m\neq 0$, $k\neq l$  and $k\neq l'$,  the indices $k,l,l'$ are pairwise different;
\item if  $m\neq 0$ and $l=l'$, then $T_k(l,k-l,l,0)= |\varphi_k|^2 |\varphi_l|^4 $;
\item if  $m\neq 0$ and $l'=k$, then $T_k(l, k-l, k, 0) = |\varphi_k|^4 |\varphi_l|^2 $;
\item if $m=m'=0$ and $l'\neq k$, then $T_k(k,0,l',0)=  |\varphi_k|^4 |\varphi_{l'}|^2 $;
\item if $m=m'\neq 0$ and $l=l'\neq k$, then $T_k(l,k-l,l,k-l) = |\varphi_k|^2 |\varphi_l|^4 $;
\item finally, if $m=m'=0$ and $l'= k$, then $T_k(k,0,k,0) =  |\varphi_k|^6 $. 
\end{itemize} 

{\it Case 1.2}\;   If $l+m=l'+m'$, we deduce that $k=l+m=l'+m'$ and hence $k-m'=l'$. Therefore, $T_k(l,k-l,l',k-l')= |\varphi_k|^2 |\varphi_l|^2 |\varphi_{l'}|^2$. Thus, 
\begin{itemize}
\item if $m\neq 0$, $m'\neq 0$ and $l\neq l'$, the indices $k,l,l'$ are pairwise different; 
\item if $m=0$ and $m'\neq 0$, then $T_k(k,0,l',k-l')=  |\varphi_k|^4 |\varphi_l'|^2 $.
\end{itemize} 
In the other cases, that is $m\neq 0$ and $m'=0$ (resp. $m=m'=0$), we have again $T_k(l,k-l,k,0)$ (resp. $T_k(k,0,k,0)$).
\bigskip

\noindent {\it Case 2} \; Let $k-m=k-m'$, which implies $m=m'$.  Then we have to impose either $l+m=l$ or $l+m=l'+m'$. 

{\it Case 2.1} \;  If $l+m=l$, we deduce $m=m'=0$, so that $l'=l'+m'$. Therefore, $T_k(l,0, l',0)=|\varphi_k|^2 |\varphi_l|^2 |\varphi_{l'}|^2$.  Thus
\begin{itemize}
\item the indices $k,l,l'$ can be pairwise distinct;
\item if $k=l'\neq l$, then $T_k(l,0,k,0)= |\varphi_k|^4 |\varphi_l|^2 $;
\end{itemize}
In the other cases $k=l\neq l'$ (resp.$ k=l=l'$), we have again $T_k(k,0,l',0)$ (resp.  $T_k(k,0,k,0)$). 

{\it Case 2.2}\;  If $l+m=l'+m'$, since $m=m'$ we conclude that $l=l'$. Therefore, $T_k(l,m,l, m)=|\varphi_{k-m}|^2 |\varphi_{l+m}|^2 |\varphi_{l}|^2$. Thus, 
\begin{itemize}
\item if  $m\neq 0$, $l\neq k-m$ and $l\neq l-2m$, the indices $k-m, l, l+m$ are pairwise distinct; 
\item if $m\neq 0$ and  $l+m=k-m$, then $l\neq k-m$ and $T_k(k-2m, m, k-2m,m)= |\varphi_{k-m}|^4 |\varphi_{k-2m}|^2$; 
\item if $m=0$ and $l\neq k$, then $T_k(l,0,l,0) = |\varphi_k|^2 |\varphi_l|^4$; 
\end{itemize}
Finally, when $m=0$ and $k=l$ (resp. $m\neq 0$ and $l=k-m$)  we have again $T_k(k,0,k,0)$ (resp. $T_k(l, k-l, l, k-l)$). 
\bigskip

\noindent{\it Case 3} Let $k-m=l'+m'$; then we have to impose either $l+m=l$ or $l+m= k-m'$.

{\it Case 3.1}\;  If $l+m=l$, that is $m=0$, we have $k=l'+m'$. Therefore, $T_k(l,0,l',k-l')=|\varphi_k|^2 |\varphi_l|^2 |\varphi_{l'}|^2$. Thus,
\begin{itemize}
\item the indices $k,l,l'$ can be pairwise distinct;
\item if $k\neq l=l'$, then $T_k(l,0,l,k-l)= |\varphi_k|^2 |\varphi_l|^4$.
\end{itemize}
When $k=l\neq l'$ (resp. $k=l'\neq l$), we have again $T_k(k,0,l',k-l')$ (resp. $T_k(l,0,k,0)$).\\
 Finally, if $k=l=l'$, we have again $T_k(k,0,k,0)$. 

{\it Case 3.2} \;  If  $l+m=k-m'$, we deduce that $m+m'=k-l=k-l'$, so that  $l'=l$. Therefore, $T_k(l,m,l,k-m-l)=|\varphi_{k-m}|^2 |\varphi_{l+m}|^2 |\varphi_{l}|^2$. Thus
\begin{itemize}
\item if $m\neq 0$, $l\neq k-m$ and $l \neq k-2m$, the indices $k-m, l ,l+m$ are pairwise distinct; 
\item if $m\neq 0$ and $l=k-m$, then $T_k(l,k-l,l,0)=|\varphi_{l}|^4 |\varphi_{k}|^2 $.
\end{itemize} 
The other cases $m\neq 0$ and $l+m=k-m$ (resp. $m=0$ and $l\neq k$) give again $T_k(k-2m,m,k-2m,m)$ (resp. $T_k(l,0,l,k-l)$. 
Finally, the case $m=0$ and $l=k$ gives again $T_k(k,0,k,0)$. 
\smallskip

As a summary, we have proven that
\begin{align} \label{B_k^2p}
 E_{\mu_a}&  (| B_k(\varphi) |^{2p})= E[|\varphi_k|^{6p}] 
+ 4 \sum_{\stackrel{|l|\leq N, |l'|\leq N}{ l\neq k, l'\neq k, l\neq l'}}  E[ |\varphi_k|^{2p}] E[ |\varphi_l|^{2p}] E[ | \varphi_{l'}|^{2p}] \nonumber \\
&+ 2\;    \sum_{\stackrel{|l|\neq N, l\neq k,  1\leq |m|\leq N} { l+m\neq k-m}} E[ |\varphi_{k-m}|^{2p}] E[ |\varphi_{l+m} |^{2p}] E[ | \varphi_{l}|^{2p}] \nonumber  \\
&+4 \sum_{|l|< N, l\neq k} E[ |\varphi_k|^{4p}] E[ |\varphi_l|^{2p}] +4  \sum_{|l|\leq N, l\neq k}  E[ |\varphi_k|^{2p}] E[ |\varphi_l|^{4p}] \nonumber \\
&+ \sum_{\stackrel{1\leq |m| \leq N}{ m\neq k, 2m\neq k}}\!\! E[ |\varphi_{k-m}|^{2p}]E[ | \varphi_{k-2m}|^{4p}] .
\end{align} 

We first deal with $p=1$.  Set 
\[ D_N:=\sum_{l:0<|l|\leq N} \int |\varphi_l|^2 d\mu_a.\] 
The above discussion implies that 
\begin{equation}  \label{EB_k^2} 
E_{\mu_a} [ |B_k(\varphi)| ^2] = 4 D_N^2 E_{\mu_a}[|\varphi_k|^2] + 4 D_N E_{\mu_a} [|\varphi_k|^4] + {\mathcal I}_{N,1}(k) + {\mathcal I}_{N,2}(k), \end{equation}
where 
\begin{align*}
{\mathcal I}_{N,1}(k)  = &\,  E_{\mu_a}[|\varphi_k|^6] +  4 E_{\mu_a}[ |\varphi_k|^2]  \sum_{0<|l|\leq N, l\neq k} E_{\mu_a}[|\varphi_l|^4] \\
&\, + { \sum_{ \stackrel{0<|m|\leq N}{ (k-m)(k-2m)\neq 0 } } } \!\! E_{\mu_a} [|\varphi_{k-m}|^4 ] E[ |\varphi_{k-2m}|^2],\\
{\mathcal I}_{N,2}(k) =  
&\,  2  \;   { \sum_{\stackrel{0<|m|\leq N, 0<|l|\leq N, (k-m)(l+m)\neq 0}  {  l+m\neq k, l+2m\neq k}} } E_{\mu_a} [|\varphi_{k-m}|^2] E_{\mu_a} [|\varphi_{l}|^2] E_{\mu_a} [|\varphi_{l+m}|^2] . 
\end{align*}

 In dimension 1, it is easy to see that  $\sup_N [D_N+{\mathcal I}_{N,1}(k)+{\mathcal I}_{N,2}(k) ] <\infty$.  Indeed, let us consider for instance ${\mathcal I}_{N,2}(k)$. { 
 Notice that 
 $\sup_{m\neq k}  E_{\mu_a} [|\varphi_{k-m}|^2]<\infty $}
  and $\sum_{\stackrel{0<|m|\leq N} {  l+m\neq 0}} E_{\mu_a} [|\varphi_{l+m}|^2] \leq C$, for a constant $C$ not depending on $l$ and $N$. 
 Then the result follows from the fact that  $\sup_N \sum_{0<|l|\leq N} E_{\mu_a} [|\varphi_{l}|^2] <\infty $. 
\smallskip

In dimension 2,  the identity \eqref{formvar} and the upper estimate \eqref{convseries2} in Lemma \ref{convseries} imply that
for $\beta <0$ we have 
\[ \sup_N \sum_{0<|k|\leq N} |k|^{2\beta} {\mathcal I}_{N,1}(k) <\infty,\]
and the upper estimate \eqref{convseries1} implies 
\[ \sup_N \sum_{0<|k|\leq N} |k|^{2\beta} {\mathcal I}_{N,2}(k) <\infty.\]

A similar (easier) computation, based on $m=0$ {or  $ l=k-m\neq k$},  implies that for $\tilde{T}_k(l,m)= \varphi_{k-m} \varphi_{l+m} \overline{\varphi_l} \overline{\varphi_k}$,  we have 
\begin{align} \label{tildeT}
 \sum_{ |l|,|m|\leq N} E_{\mu_a}[ \mbox{\rm Re }( \tilde{T}_k(l,m) ]=& E[|\varphi_k|^4] +2  \sum_{1\leq |l|\leq N, l\neq k}  E[|\varphi_k|^2] E[|\varphi_l|^2]. 
\end{align}

Using \eqref{EB_k^2} and \eqref{tildeT} we deduce that 
\begin{align*}
E_{\mu_a} &[|B_k (\varphi)-C_N \varphi_k|^{2}] =
 4 D_N^2  E[ |\varphi_k|^{2}] +4   D_N E[ |\varphi_k|^{4}] +  {\mathcal I}_{N,1} (k) \\
& + {\mathcal I}_{N,2}(k)  -2 C_N  E[|\varphi_k|^4]  
 -2 C_N  E[|\varphi_k|^2] 2 D_N 
+ C_N^2 E[|\varphi_k|^2]. 
\end{align*} 
Therefore, if $C_N=2D_N$, we have 
\[  
E_{\mu_a} [|B_k (\varphi)-C_N \varphi_k|^{2}] =  {\mathcal I}_{N,1}(k) + {\mathcal I}_{N,2}(k) , \] 
and, for $\beta <0$, 
\[ \sup_N \sum_{0<|k| \leq N} |k|^{-2\beta} E_{\mu_a} [|B_k (\varphi)-C_N \varphi_k|^{2}] <\infty,\]
 which proves that $:B:\; \in L^2_{\mu_a}(H^\beta , H^\beta)$.
\smallskip

 For more general integers $p$, we proceed recursively. Indeed, we have formally (to simplify computations, we use erroneously the independence of 
 $ |B_k (\varphi)-C_N \varphi_k|^{2p} $ and $|B_k|^2 - C_N^2 |\varphi_k|^2$,  but the  argument can be made rigorous)  
\begin{align*}
&E_{\mu_a}\big(  |B_k (\varphi)-C_N \varphi_k|^{2p+2}\big) \\
&=  E_{\mu_a} \big( |B_k (\varphi)-C_N \varphi_k|^{2p} [|B_k (\varphi)|^2 -2 C_N Re( \varphi_k \overline{B_k} )+ C_N^2 |\varphi_k|^2] \big)\\
&=  E_{\mu_a} \big( |B_k (\varphi)-C_N \varphi_k|^{2p} [|B_k (\varphi)|^2 - C_N^2 |\varphi_k|^2] \big)\\
&=  E_{\mu_a} \big( |B_k (\varphi)-C_N \varphi_k|^{2p} [  {\mathcal I}_{N,1}(k) + {\mathcal I}_{N,2}(k)  ] \big)\\
&\leq C   E_{\mu_a} \big( |B_k (\varphi)-C_N \varphi_k|^{2p} \big).
\end{align*}
This concludes the proof.

\end{proof}
Next, we turn to estimate the first derivative of $:B:$. We have
\begin{prop}
 Let $d=1,2$. For $a=2$ and $\beta <0$, the gradient of the vector field $B$ for $d=1$ (resp. $:B:$ for $d=2$) belongs to $L^{2p}_{\mu_a} (H^\beta ; H.S. (H^1 ,H^\beta))$ for any integer $p\geq 1$.

\end{prop}
\begin{proof}
As in the previous proof, we work at fixed $N$ but we drop the subscript $N$ in $(:B:)^N$.
A straight-forward computation gives  
\begin{align*}
D_{e_j} B_k (\varphi)&
= \sum_{l,m} (\varphi_{l+m} \overline{\varphi_l} \delta_{j,k-m}+ \varphi_{k-m} \overline{\varphi_l} \delta_{j,l+m}+ \varphi_{l+m} \varphi_{k-m} \delta_{-j,l}) \\
&=  \sum_{l\notin \{ 0, j-k\}} \varphi_{l+k-j} ( 2\overline{\varphi_l} + \varphi_{-l}).  
\end{align*}
Thus, we have
$$|D_{e_j} B_k (\varphi)|^2 = \sum_{l,l^\prime \notin \{ 0, j-k\}} \varphi_{l+k-j} (2 \overline{\varphi_l}+ \varphi_l )  \overline{\varphi_{l^\prime+k-j}} (2\varphi_{l^\prime}+ \overline{\varphi_{l^\prime}}) .$$

Let us first consider  the $1$-dimensional case. Since
$$\|\nabla B(\varphi) (\hat{e}_j)\|_{H^\beta}^2 \leq c \sum_{k} \dfrac{|k|^{2\beta}}{|j|^{2}}\;  \Big| \sum_{l\notin \{ 0, j-k\}}
 \varphi_{l+k-j}  (2 \overline{\varphi_l} + \varphi_{-l})  \Big|^2, $$
 we deduce
\begin{align*}
E_{\mu_a} \big[ \|\nabla B (\varphi)\|^2_{H.S. (H^1 ,H^\beta) }\big] &\leq  c \sum_{k,j} \dfrac{|k|^{2\beta}}{|j|^{2}} \;  E_{\mu_a} \Big[ \sum_{l,l'\notin \{0,  j-k\}}   T_{j,k}(l,l')\Big],
\end{align*}
 where $T_{j,k}(l,l')=\varphi_{l+k-j}  (2\overline{\varphi_l} +\varphi_{-l}) \overline{\varphi_{l'+k-j}} (2 \varphi_{l'} + \overline{\varphi_{-l'}}) $.
\bigskip

\noindent {\it Case 1:} $l=l'$, $T_{j,k}(l,l) = |\varphi_{l+k-j}|^2 (4 |\varphi_l|^2 + 2 \overline{\varphi_{l}} \overline{\varphi_{-l}} + 2 \varphi_{-l} \varphi_l  + |\varphi_{-l}|^2)$. 

\;  $\bullet$ If $k=j$, $E_{\mu_a}[T_{j,j}(l,l) ]=4  E_{\mu_a}[  |\varphi_l|^4 ] + E_{\mu_a} [|\varphi_l|^2 ] E_{\mu_a} [ |\varphi_{-l}|^2]$. 

\;  $\bullet$ If $k\neq j$ and $2l+k-j\neq 0$, $E_{\mu_a}[T_{j,k}(l,l) ] = 4 E_{\mu_a}[|\varphi_{l+k-j}|^2] E_{\mu_a}[ |\varphi_l|^2 ] + E_{\mu_a}[|\varphi_{l+k-j}|^2]  E_{\mu_a} [|\varphi_{-l}|^2] $. 

\;  $\bullet$ If  $k\neq j$ and $2l+k-j=  0$, $E_{\mu_a}[T_{j,k}(l,l) ]=  4 E_{\mu_a}[|\varphi_{l+k-j}|^2] E_{\mu_a}[ |\varphi_l|^2 ]  + E_{\mu_a}[|\varphi_{l+k-j}|^4]$.  
\bigskip

\noindent{\it Case 2:} $l\neq l'$. 

\indent $\bullet$ If $k=j$,  $T_{j,j}(l,l') = 4  |\varphi_l|^2   |\varphi_{l'}|^2 + 2 |\varphi_{l'}|^2 \varphi_l  \varphi_{-l}  + 2 |\varphi_l|^2 \overline{\varphi_{l'}} \overline {\varphi_{-l'}}
+ \overline{ \varphi_{l'}} \overline{ \varphi_{-l'}}$.  Therefore,
$E_{\mu_a}[ T_{j,j}(l,l') ] = 4  E_{\mu_a} [ |\varphi_l|^2]  E_{\mu_a}[  |\varphi_{l'}|^2]$.  

\indent $\bullet$ If $k\neq j$, $E_{\mu_a}[T_{j,k}(l,l')]=0$, except if $ l+k-j=-l'$. In that case we have
$E_{\mu_a}[T_{j,k}(l,l')] = E_{\mu_a} [|\varphi_{l+k-j}|^2] E_{\mu_a} [|\varphi_{-l}|^2]$ if $l+k-j\neq -l$, and $E_{\mu_a} [|\varphi_{l+k-j}|^4]$ if
$l+k-j=  -l$. 
\bigskip

 For $a=2$ and $\beta <0$, using once more \eqref{convseries1}, 
  we have 
\begin{align*}
\sum_{k\neq 0} &\frac{|k|^{2\beta}}{|k|^2} \Big[ \sum_{l\neq 0} |l|^{-2a} + \sum_{l,l' \neq 0, l\neq l'} |l|^{-a} |l'|^{-a} \Big] \\
& + 
 \sum_{k,j\neq 0, k\neq j} \dfrac{|k|^{2\beta}}{|j|^{2}}  \sum_{l\notin \{0, j-k\}}\big(  |l|^{-a}|l+k-j|^{-a} +
 |l+k-j|^{-2a}\big) <\infty.
\end{align*}

Next, we deal with the $2$-dimensional case. We need to be more careful in this case and we begin by computing $\|D_{e_j} B_k (\varphi)\|_{L^2_{\mu_a}}$. As in the one-dimensional case, we have
\begin{align*}
\int & |D_{e_j}B_k(\varphi)|^2 d\mu_a =   \delta_{j,k} \Big(  4 \sum_{l\neq 0}  E_{\mu_a}[|\varphi_l|^4] + \sum_{l\neq 0} E_{\mu_a}[|\varphi_l|^2]  E_{\mu_a}[|\varphi_{-l}|^2]\\
& \qquad \qquad + 4 \sum_{l,l'\neq 0, l\neq l'} E_{\mu_a} [|\varphi_l|^2] E_{\mu_a} [|\varphi_{l'}|^2|] \Big)\\
& + 1_{\{j\neq k\} } \Big( 4 \sum_{l  \notin \{0, j-k\} } E_{\mu_a}[ |\varphi_{l+k-j}|^2]  E_{\mu_a}[ |\varphi_l|^2]  +2 1_{\{ j-k \; {\rm even}\}} E_{\mu_a} [ |\varphi_{(k-j)/2}|^4]  \\
& \qquad \qquad +2  { \sum_{l,2l \notin \{0, j-k\}} } E_{\mu_a}[ |\varphi_{l+k-j}|^2]  E_{\mu_a}[ |\varphi_{-l}|^2] \Big).
\end{align*}
Let 
\begin{align*}
 {\mathcal J}_N&(j,k)= \delta_{j,k} \Big(  4 \sum_{l\neq 0}  E_{\mu_a}[|\varphi_l|^4] + \sum_{l\neq 0} E_{\mu_a}[|\varphi_l|^2]  E_{\mu_a}[|\varphi_{-l}|^2] \Big) \\
&+  1_{\{j\neq k\} }\Big(  4 \sum_{l  \notin \{0, j-k\} } E_{\mu_a}[ |\varphi_{l+k-j}|^2]  E_{\mu_a}[ |\varphi_l|^2]  +2 1_{\{ j-k \; {\rm even}\}} E_{\mu_a} [ |\varphi_{(k-j)/2}|^4]  \\
& \qquad \qquad +2 {  \sum_{l,2l \notin \{0, j-k\}} } E_{\mu_a}[ |\varphi_{l+k-j}|^2]  E_{\mu_a}[ |\varphi_{-l}|^2] \Big).
\end{align*}

Using \eqref{:B:}, we deduce that 
$$D_{e_j} (:B :)_k (\varphi) = D_{e_j} B_k (\varphi) - C_N \delta_{j,k} .$$
Therefore, since $C_N=2 D_N$, 
\begin{align*}
\int & |D_{e_j}B_k(\varphi)|^2 d\mu_a =   {\mathcal J}_N(j,k) + \delta_{j,k}\big( 4 D_N^2 - 2C_N \, 2D_N + C_N^2\big)  = {\mathcal J}_N(j,k) . 
\end{align*}
Using \eqref{formvar} and \eqref{convseries1} in Lemma \ref{convseries}, we deduce that
$$\sup_N \sum_{j,k} \frac{|k|^{2\beta}}{|j|^2} {\mathcal J}_N(j,k) <\infty.$$
For more general integers $p$, we proceed recursively. Indeed, we have formally (as previously, we use erroneously the independence
 of $ |D_{e_j}B_k(\varphi)|^2$ and $|D_{e_j}B_k(\varphi)|^{2p}$, but the argument can be made rigorous)  
\begin{align*}
E_{\mu_a}\big(  |D_{e_j}B_k(\varphi)|^{2p+2}\big) 
&=  E_{\mu_a} \big( |D_{e_j}B_k(\varphi)|^{2p} |D_{e_j}B_k(\varphi)|^2 \big)\\
&=  E_{\mu_a} \big( |D_{e_j}B_k(\varphi)|^{2p}  {\mathcal J}_N(j,k)  \big)\\
&\leq C   E_{\mu_a} \big( |D_{e_j}B_k(\varphi) |^{2p} \big).
\end{align*}
This concludes the proof.
\end{proof}

Using the definition of the second derivatives and corresponding norms presented in section 2, we have the following result:

\begin{prop}
Let $a=2$, $\beta<0$; the second derivative of $B$ belongs to $L^{2p}_{\mu_a} (H^\beta , H.S. (H^1 \otimes H^1 , H^\beta))$, for any integer $p\geq 1$.

\end{prop}

\begin{proof}
A straight-forward computation gives
$$\| D_{\hat{e}_l}D_{\hat{e}_m} B (\varphi) \|^2_{H^\beta}\lesssim \dfrac{1}{(|l||m|)^{2}} \sum_k |k|^{2\beta}  (|\varphi_{l+m-k}|^2 + |\varphi_{l-m+k}|^2 ) . $$

By definition of the Hilbert-Schmidt norm, this leads to
\begin{align*}
 \|\nabla^2 B(\varphi) \|_{H.S.(H^1 \times H^1 , H^\beta)}^2 \lesssim &\, \sum_{ l,m\neq 0}  \dfrac{1}{(|l||m|)^{2}}  \\
&\, \times \sum_{ k\not\in \{ 0, m+l,m-l\}  }
 |k|^{2\beta} ( |\varphi_{l+m-k}|^2 + |\varphi_{l-m+k}|^2 ) . 
 \end{align*}
Using once more \eqref{formvar}, we deduce that
$$E_{\mu_a} \|\nabla^2 B(\varphi) \|_{H.S.(H^1 \times H^1 , H^\beta)}^2 \lesssim  \sum_{ l,m\neq 0}  \dfrac{1}{(|l||m|)^{2}} \sum_{ k\not\in \{ 0, m+l,m-l\}  }
 |k|^{2\beta}  \dfrac{1}{|l+m-k|^a}.  $$
This last term converges when $\beta <0 $ (see \eqref{convseries1} in  Lemma \ref{convseries}). As previously, the case for more general integers $p$ can be made by induction.
\end{proof}

More generally, the following holds:

\begin{prop}  For any $ s\in \N$ and any integer $p\geq 1$, 
 $\nabla^s :B:$ belong to $L^p_{\mu_a} (H^\beta ; H.S. (\otimes^{s} H^1 , H^\beta))$    when $a=2$ and   $\beta<0$. 
\end{prop}

\section{The renormalised mass} \label{renormalized}
We will consider level surfaces for the  renormalised mass $E$, where
$$E(\varphi)=\sum_k ( |\varphi_k|^2 -Z_k).$$
Let us recall  that the mass is an invariant of the NLS flow. We choose $Z_k$ such that $E$ is integrable with respect to $\mu_a$, i.e.,  $Z_k = c_{1} |k|^{-a}$ where $c_p$ was defined
 in \eqref{formvar}. As for $B$, this renormalisation is not needed in dimension $1$ since the series  is convergent.
  We recall that $W_{\infty} = \bigcap_{r,p} W_{r}^p (H^\beta , \mathbb{R})$.

Our goal in this section is to show that $E$ and $\|\nabla E\|^{-1}_{H.S. ( H^l ,  \mathbb R) }$ 
belong to the space $W_\infty$,
 as it was defined before, namely for $a=2$ and $l=1$ (let us recall that the norm of this space depends on $\mu_a$). 
 As previously, our choice of parameters are linked with the fact that we want to work on the abstract Wiener space $(H^\beta ,H^1 ,\mu_2)$, $\beta<0$. We begin by estimating $E$.

\begin{prop} \label{Prop_EW}
The mass $E$ belongs to $W_{\infty}$.

\end{prop}

\begin{proof}
To ease notations all the sums below will be considered for indices $k$ and $l$ different from 0.  A straight-forward computation gives
\begin{align*}
&\int |E(\varphi)|^2 d\mu_a \\
&= \int \sum_k (|\varphi_k|^2 -Z_k)^2 d\mu_a+ \sum_{k\neq l} \int |\varphi_k|^2 |\varphi_l|^2 d\mu_a - 2 \sum_{k\neq l} \int Z_l |\varphi_k|^2 d\mu_a + \sum_{k\neq l} Z_k Z_l\\
&= \int \sum_k |\varphi_k|^4 d\mu_a - 2 \int \sum_k Z_k |\varphi_k|^2 d\mu_a + \sum_k Z_k^2 \\
&\qquad + \sum_{k\neq l} \int |\varphi_k|^2 |\varphi_l|^2 d\mu_a - 2 \sum_{k\neq l} \int Z_l |\varphi_k|^2 d\mu_a + \sum_{k\neq l} Z_k Z_l\\
&= \sum_k c_2 |k|^{-2a}  -2 \sum_k Z_k c_1 |k|^{-a} + \sum_k Z_k^2 \\
&\qquad +  \sum_{k\neq l} c_1 |k|^{-a}  c_1 |l|^{-a}  -2 \sum_{k\neq l} Z_l c_1 |k|^{-a} + \sum_{k\neq l} Z_k Z_l\\
&= \sum_k  (c_2 |k|^{-2a} - Z_k^2 ) \lesssim  \sum_k  |k|^{-2a} <\infty .
\end{align*}
To deal with a general integer $p\geq 2$, using the independence of the $\varphi_{k_j}, j=1, ..., p$ 
and an induction argument, we have
\begin{align*}
\int |E(\varphi)|^p d\mu_a =& \int \sum_k (|\varphi_k|^2 -Z_k)^p d\mu_a + \sum_{k_1\neq k_j} \int ( |\varphi_{k_1}|^2 -Z_{k_1}) \Pi_{k_j} ( |\varphi_{k_j}|^2 -Z_{k_j}) d\mu_a
\\
= & \int \sum_k (|\varphi_k|^2 -Z_k)^p d\mu_a + \sum_{k_1\neq k_j}  \int  |\varphi_{k_1}|^2 \Pi_{k_j}( |\varphi_{k_j}|^2 -Z_{k_j})d\mu_a 
\\
& - \sum_{k_1\neq k_j} \int  Z_{k_1} \Pi_{k_j} ( |\varphi_{k_j}|^2 -Z_{k_j}) d\mu_a  \\
\lesssim & \sum_k |k|^{-ap}+  \sum_{k_1\neq k_j}  \int  Z_{k_1} \Pi_{k_j} ( |\varphi_{k_j}|^2 -Z_{k_j}) d\mu_a
  \\
& - \sum_{k_1\neq k_j} \Pi_{k_j} \int  Z_{k_1} ( |\varphi_{k_j}|^2 -Z_{k_j}) d\mu_a\\
\lesssim &\sum_k |k|^{-ap}.
\end{align*}

Next, we have
$$\nabla E (\varphi)(e_k)= 2 Re (\varphi_k ).$$
Therefore, we deduce 
$$\| \nabla E(\varphi) \|_{L_{\mu_a}^{2m} ( H.S. ( H^1 , \mathbb R ) )}^{1/m} \leq c \sum_{k} \big( E_{\mu_a} \big[  |k|^{-2m} |\varphi_k|^{2m}\big] \big)^{1/m} 
\leq c(p) \sum_{k} \dfrac{1}{|k|^{a+2}}<\infty. $$
Finally, we also have
$$\| \nabla^2 E(\varphi) \|_{L_{\mu_2}^{2m} ( H.S. ( H^1 \otimes H^1 ,  \mathbb R) }^{1/m} \leq C   \sum_k |k|^{ -4  } <\infty , $$
which completes the proof since higher derivatives vanish.
\end{proof}

We now turn to the estimate of $\|\nabla E\|^{-1}_{H.S. ( H^1 , \mathbb R) }$.

\begin{prop} \label{Prop_E^-1W}
The random variable $\|\nabla E\|^{-1}_{H.S. ( H^1 ,  \mathbb R) }$ belongs to the space $W_\infty$.
\end{prop}
\begin{proof}
By Chebycheff's  inequality, we have, for all $t>0$ and $\varepsilon >0$,
\begin{align*}
\mu_a \{\|\nabla E (\varphi)\|_{H.S. ( H^1 ,  \mathbb R ) }^2 \leq \varepsilon\}& \leq e^{t\varepsilon} E_{\mu_a} \big[ e^{-t \|\nabla E(\varphi) \|_{H.S. ( H^1 , \R) }^2 }\big]\\
&\leq C e^{t\varepsilon}\prod_{k\neq 0} |k|^a \int_0^\infty e^{- 4 t |k|^{-2} r^2 - \frac{|k|^{a}}{2} r^2} rdr\\
&\leq C e^{t\varepsilon} \prod_{k\neq 0}  \dfrac{1}{1+ 8t/|k|^{a+2}}.
\end{align*}
On the other hand, we have 
\begin{align*}
\prod_{k\neq 0}   \Big(1+ \dfrac{8t}{|k|^{a+2}}\Big)& \geq \prod_{k\neq 0, |k|\leq (8t)^{1/(a+2)}}   \Big(1+ \dfrac{8t}{|k|^{a+2}}\Big) \\
&\geq 2^{(8t)^{\frac{1}{a+2}}} = e^{{\color{magenta} \tilde{c} }  t^{\frac{1}{a+2}}},
\end{align*}
 for $\tilde{c} =8^{\frac{1}{a+2}} \ln(2) <1$ when $a\geq 2$. Therefore, we deduce that, for some $c>0$, we have  
\begin{align*}
\mu_a \{\|\nabla E (\varphi)\|_{H.S. ( H^1 , \R) }^2   \leq  \varepsilon\}& \leq  \inf_{t>0}  e^{t\varepsilon - { \tilde{c} }  t^{\frac{1}{a+2}} }
=e^{-c \varepsilon^{\frac{1}{1-(a+2)}}}, { \quad \forall \epsilon >0.}
\end{align*} 
Thus, letting $\|\nabla E (\varphi)\|_{H.S. ( H^1 ,  \mathbb R )}^2=X$ and using the equality $\mu_a (1/|X| \geq \varepsilon) = \mu_a(|X|\leq 1/\varepsilon)$, we have
 \begin{align*}
E_{\mu_a}\Big(\dfrac{1}{|X|}^p\Big) 
&= \int_0^\infty p \varepsilon^{p-1}  \mu_a \Big(\dfrac{1}{|X|} \geq \varepsilon\Big) d\varepsilon \\
&  \leq \int_0^\infty p \varepsilon^{p-1}  e^{-c \varepsilon^{\frac{1}{(a+2)-1}}  }d\varepsilon <\infty 
\end{align*}
when $1<a+2$ and, in particular, for  $a=2$.
\end{proof}

\section{The NLS equation and invariant measures}
Following \cite{B96}, if we consider the renormalised NLS equation and its corresponding renormalised Hamiltonian, which is the limit in $N$ of

$$ H_N(\varphi)  =\int_{\T^2}  |\nabla \varphi|^2 +\frac{1}{2}  \int_{\T^2}  |\varphi|^4 -2 a_N   \int_{\T^2}  |\varphi|^2 +a_N^2 ,$$
with $a_N =2 \sum_{k\neq 0 :|k|<N}   \frac{1}{|k|^2} $. The corresponding Gibbs measure 
$$d\mu (\varphi) =\lim_N  e^{-{\frac{1}{2} } \int_{\T^2}  |\varphi|^4 +2 a_N  \int_{\T^2} |\varphi|^2 -a_N^2}     d\mu_2(\varphi)$$
is invariant for the renormalised flow \eqref{rNLS},  and the measures $e^{-H_N (\varphi)} ~\Pi d\varphi$ are invariant for the corresponding truncated equations. Moreover $\mu$ 
has a Radon-Nikodym derivative w.r.t. $\mu_2$ which belongs to all $L^p$ spaces (all these statements were shown in \cite{B96}).

Therefore, by Theorem 2.1 and thanks to Propositions  \ref{Prop_EW}  and \ref{Prop_E^-1W}, 
 we can repeat the construction of the surface measures on   level sets $(V_r)$ of the mass $E$,  
  starting from this weighted Wiener measure $\mu$ instead of $\mu_2$. 
 Let us denote by $\sigma_r$ this surface measure on  the level set of renormalised  mass equal to $r$.

In this section we prove  existence and uniqueness  of a global solution of the renormalised cubic nonlinear Schr\"odinger equation in the two-dimensional torus and 
 living in the level sets of the renormalised mass.
More precisely, we have

\begin{thm}		\label{Th5.1}
For each $r$ such that $\rho (r)=\frac{d(E*\mu)}{dr}>0$ there exists a probability space $(\Omega , \mathbb F, P)$ and a unique flow $u \in C(\mathbb R \times V_r )$, 
such that for every $t\in \R$ and $\varphi \in V_r$,

$$(i) \quad u_t ( \varphi ) =e^{itA}   \varphi   + \int_0^t e^{-i(t-s)A} :B:^\ast  (u_s (\varphi ))ds, \qquad \sigma_r-a.s. $$
where $ :B:^\ast$ is a redifinition of $:B:$ as defined in Section $2.4$. Moreover the measure $\sigma_r$ is invariant for the flow, namely

$$(ii) \quad \int_{V_r}  f(u_t ( \varphi)) d\sigma_r (\varphi)= \int_{V_r} f (\varphi) d\sigma_r (\varphi), $$
$\hbox{for every cylindrical}~f~\hbox{functional  and every}~t.$

\end{thm}

 \begin{proof}

 The beginning of the proof follows the lines of  
 \cite[Theorem 4.1]{C}  for the Euler equation.

First notice that, given the estimates of all Sobolev norms of the field $:B:$ which were previously derived, one can consider a redefinition $:B:^\ast$ on each level set $V_r$. By the definition of the surface measure, $:B:^\ast$ belongs also to all Sobolev spaces with respect to this measure. 

From the regularity of the finite dimensional approximations $:B:^n$ for every $n$, 
we know that there exists a global solution $v_t^n (\varphi )$ of the equation, written here in its integral form,

\begin{equation} \label{vnphi}
v_t^n (\varphi )= e^{itA}\varphi +\int_0^t e^{-i(t-s)A} :B:^n (v_s^n (\varphi ))ds~, \; t \in \R, \; \varphi \in H^{\beta}.
\end{equation}

For every cylindrical functional $f$ and every $h\in {C^1}(\mathbb R)$ with compact support we have, writiting with the same notation the scalar product
\linebreak 
$\langle h_1 , h_2 \rangle_{H^1}$ and its extension to $h_1 \in H^\beta ,h_2 \in H^1$,

\begin{align*}
\int_{\mathbb R} h(r) \rho (r)& \int_{V_r} \langle (A- :B:^n )(\varphi ), \nabla f (\varphi )\rangle_{H^1} d\sigma_r (\varphi ) dr\\
&=\int_{H^{\beta}} \langle h((E (\varphi )) (A- :B:^n  )(\varphi ), \nabla f (\varphi )\rangle_{H^1}   d\mu (\varphi )\\
&=-\int_{H^{\beta}}  \langle h' (E (\varphi )) (A- :B:^n ), \nabla E (\varphi ) \rangle_{H^1} f (\varphi)  d\mu (\varphi )\\
&=0,
\end{align*}
where the last equality is a consequence of the invariance of the mass.

This implies the invariance of the surface measure with respect to the  flows $v_t^n(.)$ or, in other words, that the divergence (the $L_{\sigma_r}^2$-dual 
of the gradient) of $A-:B:^n$ is equal to zero.

Denote by $Q_n^k$ the law of $[v_t^n ]^k$ in the space $C(\mathbb R^+ ; \mathbb C )$  endowed with the supremum norm, namely

$$Q_n^k (\Gamma ) =\sigma_r(\{ [v^n (\varphi )]^k \in \Gamma, \quad \Gamma \subset C(\mathbb R^+ ; \mathbb C ) \} ).$$
The index $k$ stands here for  component in the $H^\beta$ basis  $\{ \frac{1}{k^\beta} e_k \} $.

On the space of probability measures over the real valued continuous path space we consider the weak topology. We show that, for $w\in C(\mathbb R^+ ; \mathbb C )$,
$$\lim_{R\rightarrow \infty} \sup_n Q_n^k \Big( |w(0)| >R \Big) =0,$$
and that for every $R>0$ and $T>0$
$$\lim_{\delta \rightarrow 0}\sup_n Q_n^k \Big( \sup_{0\leq t\leq t' \leq T , t'-t\leq \delta} |w(t) -w(t' )|\geq R \Big)=0 .$$ 
For negative times the proof is analogous.

The first statement is simply due to the estimate
\begin{align*}
Q_n^k \Big( |w(0)| >R \Big)& \leq  \frac{1}{R^2 } \int |w(0) |^2 dQ_n^k\\
&=\frac{1}{R^2 } \int |\varphi^k |^2 d\sigma_r(\varphi), 
\end{align*}
 where in the first inequality  we used Chebycheff's inequality. Concerning the second one, we have 
\begin{align*}
&Q_n^k \Big( \sup_{0\leq t\leq t' \leq T, t'-t\leq \delta}|w(t) -w(t' )|\geq R \Big)  \\
&\leq \frac{1}{R} \int_{V_r}  \sup_{0\leq t\leq t' \leq T, t'-t\leq \delta}| | [[v_t^n (\varphi )-v_{t'}^n (\varphi ) ]^k | d\sigma_r(\varphi)\\
&\leq \frac{1}{R}  \sup \Big(\int_{V_r}  \Big( |[e^{it' A} \varphi -e^{itA}\varphi ]^k | 
+ \Big|\Big[ \int_0^t |( e^{-i (t -s)A} ) :B:^n (v_s^n (\varphi ))]  ds \\ 
&\qquad \qquad \qquad \qquad -\int_0^{t'} ( e^{-i (t' -s)A} ) :B:^n (v_s^n (\varphi ))\Big]^k \Big| \Big) d\sigma_r(\varphi) \Big)
\end{align*}
Using equalities
$\frac{d}{dt} e^{itA}\varphi = itA e^{itA}\varphi$ and 
\begin{align*}
\frac{d}{dt}&  \int_0^t ( e^{-i (t -s)A} ) :B:^n (v_s^n (\varphi ))  ds \\
&=-iAe^{-itA} \int_0^t e^{isA} :B:^n (v_s^n (\varphi ))  ds +{ :B:^n (v_t^n (\varphi )} ),
\end{align*} 
we derive
\begin{align*}
Q_n^k &\Big( \sup_{0\leq t\leq t' \leq T, t'-t\leq \delta}|w(t) -w(t' )|\geq R \Big) \\
&\leq \frac{\delta^2}{R^2} c(k,T) \Big( \int_0^T{ \int_{V_r} } [ |e^{isA}\varphi |^k ]^2 dsd\sigma_r (\varphi )\Big)^{\frac{1}{2}}\\
&\quad +  \frac{\delta}{R} c(k,T)  \int_0^T\int_0^T { \int_{V_r} }  |e^{-isA}e^{itA}\varphi |^k |[:B:^n (\varphi )]^k |  (\varphi)ds dt d\sigma_r (\varphi )\\
& \quad +  \frac{\delta}{R} c(k,T) \int_0^T { \int_{V_r} } \int |[:B^n : (\varphi )]^k | (\varphi) ds d\sigma_r (\varphi )\\
& \leq \frac{\delta c_r}{R^2} \int_{V_r}  \| \varphi \|_{H^{\beta}} d\sigma_r (\varphi ) +  \frac{\delta c_r}{R}  \int_0^T { \int_{V_r} } |[:B :^n (  \varphi )]^k |(\varphi )ds d\sigma_r(\varphi)\\
& \leq \frac{\delta c_r}{R^2} \int_{V_r} \Big( \|\varphi\| _{H^{\beta}} + \| :B:^n \|_{H^{\beta}} (\varphi )\Big)d\sigma_r (\varphi ), 
\end{align*}
where we have denoted  $c_r$ a generic constant depending on $r, k$ and $T$,
 and the last step is a consequence of the quasi-invariance of the process with respect to the surface measure and the uniform (in time) estimates of the $L^2$ norms of their laws.
\color{black}

The limit follows from the fact that, since the sequence  $\{:B:^n\}_n$,  as well as that of  their gradients converges to $:B:$ in all $L^p _\mu$ spaces, it  also converges (to $:B:^\ast$) in $L^p_{\sigma_r}$.
Hence
$$Q_n^k \Big( \sup_{0\leq t\leq t' \leq T, t'-t\leq \delta}|w(t) -w(t' )|\geq R \Big) \leq  \frac{\delta}{R^2} Tc$$
for some constant c, independent of $n$.
We conclude that the sequence of  probability laws $Q_n^k$ is tight on  {$C({\mathbb R}^+;\mathbb{C})$,  }
and therefore that we can extract a subsequence converging weakly to a probability measure $Q$.
 By Skorohod's theorem there exists a probability space $(\Omega , \mathbb F, P)$ and processes $u_t^n (\omega )=\sum_k [u_t^n (\omega )]^k e_k , u_t (\omega )$ 
 with laws $Q_n^k$ and $Q^k$ resp., such that $u_{\cdot}^n \rightarrow u_{\cdot} $ $P$-almost surely.

Concerning (ii), for every cylindrical functional $f$, we have

\begin{align*}
 \int_{\Omega}  f(u_t ( \omega )) dP(\omega )&=\lim_n \int_\Omega  f(u_t^n (\omega ))dP(\omega )\\
 &= \lim_n \int_{V^r}  f(v_t^n (\varphi ))d\sigma_r (\varphi )\\
&=\lim_n \int_{V^r}  f(\varphi) d\sigma_r (\varphi) = \int_{V^r} f (\varphi)  d\sigma_r(\varphi).
\end{align*}
For the last equality we have used  the invariance of the measure $\sigma_r$ for the flows $v_t^n$.

\color{black}
Recalling that $\rho(r)= \dfrac{d(E\ast \mu)}{dr} >0$, we have
\begin{align*}
 \int_{\R^+} \psi (r) \rho (r) \int_{V_r} f(v_t^n (\varphi )) d\sigma_r(\varphi) dr&= \int_{V_r} \psi (E(\varphi ))f(v_t^n (\varphi ))d\mu (\varphi )\\
&= \int_{V_r}   \psi (E (v_{-t}^n (\varphi ))f(\varphi )d\mu (\varphi )\\
&= \int_{V_r}   \psi (E (\varphi ))f(\varphi )d\mu (\varphi )\\
&= \int_{\R^+}   \psi (r) \rho (r) \int_{V_r} f(\varphi) d\sigma_r(\varphi) dr
\end{align*}
for every smooth function $\psi$. 
 In particular the process $u_t$ takes values in $V_r$.

We show $(i)$. For every $k$, we get for $\beta <0$ 

\begin{align*}
\int_\Omega  \int_0^t & \Big| \Big[e^{-i(t-s)A} :B:^n (u_t^n (\omega ))- \int_0^t e^{-i(t-s)A} :B:^\ast  (u_t (\omega ))\Big]^k \Big| ds dP (\omega)\\
&\leq  \int_\Omega  \int_0^T \Big| \Big[e^{-i(t-s)A} \Big( :B:^n (u_t^n( \omega) )- :B:^\ast (u_t^n(\omega) )\Big)\Big]^k \Big| ds dP(\omega) \\
&+\int _\Omega\int_0^T \Big| \Big[e^{-i(t-s)A} \Big( :B:^\ast (u_t^n(\omega) )- :B:^\ast (u_t(\omega) )\Big)\Big]^k \Big| ds dP(\omega) \\
&\leq c(k) \Big( \int  \int_0^T \Big\| \Big( :B:^n (u_t^n(\omega) )- :B:^\ast (u_t^n(\omega) )\Big) \Big\|_{H^\beta} ds dP(\omega) \\
&+\int_\Omega \int_0^T  \Big\| \Big( :B:^\ast (u_t^n(\omega) )- :B:^\ast (u_t(\omega) )\Big)  \Big\|_{H^\beta} ds dP(\omega) \Big).
\end{align*}
The first term converges to zero due to the invariance of the measure $P$ and the ($L^2$) convergence of $:B^n :$ to $:B:^\ast$. Concerning the second term,
let $\Lambda\subset H_\beta$ be a set of $\sigma_r$ measure equal to zero such that $:B:^\ast$ is continuous in $\Lambda^c $. Define 

$$\Lambda_n =\{(t,\omega )\in [0,T]\times \Omega : u_t^n (t,\omega ) \in  \Lambda \}, $$
and
$$\Lambda_\infty =\{(t,\omega )\in [0,T]\times \Omega : u_t (t,\omega ) \in \Lambda \}.$$
We have,

$$\int_0^T \int_\Omega \chi_\Lambda (u_t^n (\omega ))  dP(\omega ) dt =\int_0^T \int_\Omega \chi_\Lambda (\varphi ) d\sigma_r (\varphi )dt =0, $$
i.e, the measure of $\Lambda_n$ (and, analogously, of $\Lambda_\infty$), considering the Lebesgue measure in the time coordinate, is equal to zero.

On the other hand, since $u_{\cdot}^n$ converges to $u_{\cdot}$ almost surely,  consider a set $D\subset \Omega$ such that $P(D^c )=0$ where this convergence holds, 
in the space $H^{\beta}$ and for every $t\in [0,T]$. Then the measure of $[0,T]\times D^c$ is zero.

Using Egorov's theorem, we conclude that, for $(t,\omega)$ in $\Lambda^c_\infty \bigcap_n \Lambda_n^c \bigcap ([0,t] \times D^c )$, and $\beta <0$ 
$$\int_\Omega \int_0^T  \Big\| \Big( :B:^\ast (u_t^n (\omega))- :B:^\ast (u_t(\omega) )\Big) \Big\|_{H^\beta} ds dP(\omega) \rightarrow 0, $$
which concludes the  proof of existence of a solution in a weak sense. Namely, there exists a stochastic process $u_t$ such that, for every $t$, 
$$u_t (\omega )=e^{itA} u_0 (\omega ) +\int_0^t e^{-i (t-s)A} :B:^\ast (u_s (\omega ))ds \quad P - a.s$$
and verifying

$$\int_{\Omega} f(u_t (\omega )) dP (\omega )= \int_{V_r} f(\varphi)  d\sigma^r(\varphi)$$
for every cylindrical functional $f$ and every $t$.

It remains to show the existence of a flow (defined $\sigma_r$- almost everywhere) and its uniqueness. For this we follow the arguments in \cite{AF}.

Consider a functional $k_0 \in W_{\infty}$ and bounded, whose redefinition on $V_r$ we still denote by $k_0$. Then, by approximating $k_0$ by cylindrical functionals,
we have

\begin{align*}
 \lim_n \int_{V_r}  f(v_t^n (\varphi ) &)k_0 (\varphi )d\sigma_r (\varphi )=\lim_n { \int_{V_r}  }  f(v_t^n (\varphi )) k_0 (v_{-t}^n(\varphi ) \circ v_t^n(\varphi) d\sigma_r (\varphi )\\
&:=\lim_n \int_{V_r} f(\varphi)  k_t^n (\varphi) d\sigma_r(\varphi) = \int_{V_r} f (\varphi)  k_t(\varphi)   d\sigma_r (\varphi),
\end{align*}
where $k_t$ is the weak $L^2$ limit of the sequence $k_t^n$. The functional $k_t $ solves the continuity equation $\frac{d}{dt} k_t =div_{\sigma_r} (k_t B) $, 
denoting by $ div_{\sigma_r} $ the $L^2_{\sigma_r}$ dual of the gradient (c.f \cite{AF}, Prop. 4.8.).
Existence and uniqueness of the flow follows as in the proof of \cite{AF} Theorem 4.7.

\end{proof}

\section{Appendix}

In this appendix, we show a convergence result for  series that were used quite often in section \ref{estim-B}.

\begin{lem}
\label{convseries} Let $d=1,2$. 
For any $\beta<0$, it holds that
\begin{equation} \label{convseries1}
\sum_{0\neq l,m\in \Z^2} \dfrac{1}{|l|^2}\dfrac{1}{|m|^2} \sum_{k\not\in \{0, l+m\}} \dfrac{|k|^{\beta}}{|l+m-k|^2}<\infty, 
\end{equation}
and
\begin{equation}  \label{convseries2}
\sum_{k\neq 0, m\not\in \{k, k/2\} } \dfrac{|k|^{\beta}}{|k-m|^4 |k-2m|^2}<\infty.
\end{equation}
\end{lem}
\begin{proof}
It is immediate to see that the above series converge in dimension $1$. Indeed, for $l$ and $m$ fixed, we have that $ \sum_{k\not\in \{0, l+m\}} \dfrac{|k|^{\beta}}{|l+m-k|^2}\leq C$, for some constant $C$ not depending on $l$ and $m$. Then the result follows from the fact that $\sum_{0\neq l\in \Z^2} \dfrac{1}{|l|^2}$ is bounded.

Let $d=2$. We first prove \eqref{convseries1}. Using a change of variable, the convergence of the serie is equivalent to the one of
$$\sum_{0\neq l,m \in \Z^2} \dfrac{1}{|l|^2}\dfrac{1}{|m|^2} \sum_{k \not\in \{0, -(l+m)\} }\dfrac{|k+l+m|^{\beta}}{|k|^2}.$$
Let $\lambda \in (0,1).$ We are going to split the sum in $k$ into three parts. 
\begin{itemize}
\item First, we take $|k|\leq (1-\lambda) |l+m|$; this yields 
$$\sum_{0 \neq k \in \Z^2,\ |k|\leq (1-\lambda) |l+m|} \dfrac{|k+l+m|^{\beta}}{|k|^2} \leq \lambda^{\beta} |l+m|^{\beta} \sum_{1\leq n \leq (1-\lambda)|l+m| } \dfrac{2}{n} \leq C_\lambda |l+m|^{\beta /2}. $$
\item Then, we consider $|k|\geq (1+\lambda)|l+m|$. In this case, we get 
\begin{align*}
\sum_{0 \neq k \in \Z^2,\ |k|\geq (1+\lambda) |l+m|} \dfrac{|k+l+m|^{\beta}}{|k|^2}& \leq C_\lambda \sum_{0 \neq k \in \Z^2,\ |k|\geq (1+\lambda) |l+m|}   |k|^{\beta - 2} \\  
& \leq C_\lambda   |l+m|^{\beta }. 
\end{align*}
\item Finally, when $(1-\lambda )|l+m| \leq |k|\leq (1+\lambda)|l+m|$, i.e.,  $|k|\approx |l+m|$, we have
\begin{align*}
\sum_{0 \neq k \in \Z^2,\ |k|\approx |l+m|} \dfrac{|k+l+m|^{\beta}}{|k|^2}& \leq C_\lambda |l+m|^{\beta /2} \sum_{0 \neq k \in \Z^2,\ |k|\approx |l+m|} |k|^{\beta/2 -2} \\
& \leq C_\lambda |l+m|^{\beta /2}. 
\end{align*}
\end{itemize}
 Therefore, we have to upper  estimate
$$\sum_{0\neq l,m \in \Z^2, l+m\neq 0} \dfrac{1}{|l|^2}\dfrac{|l+m|^{\beta/2}}{|m|^2}.$$
A similar argument implies  that
$$\sum_{0\neq l,m\in \Z^2, l+m\neq 0} \dfrac{1}{|l|^2}\dfrac{|l+m|^{\beta/2}}{|m|^2} \leq  C \, \sum_{0\neq l \in \Z^2} \dfrac{|l|^{\beta/4}}{|l|^2} <\infty, $$
which completes the proof of \eqref{convseries1}.

A  change of variables implies that
$$\sum_{k\not\in \{  0, m, 2m\} } \dfrac{|k|^{\beta}}{|k-m|^4 |k-2m|^2} = \sum_{k\not\in \{ 0, -m,m\}  } \frac{|k+m|^\beta}{|k|^4 |k-m|^2}.$$
A decomposition similar to the above one  implies \eqref{convseries2}.  
\end{proof}

\end{document}